\documentclass[graybox,envcountsect,envcountsame]{svmult}

\usepackage{mathptmx}       
\usepackage{helvet}         
\usepackage{courier}        
\usepackage{type1cm}        
\usepackage{makeidx}         
\usepackage{graphicx}        
\usepackage{multicol}        
\usepackage[bottom]{footmisc}

\usepackage{amsmath}
\usepackage{amssymb}
\newcommand{\blu}{\color{black}}
\newcommand{\blk}{\color{black}}

\makeindex             

\def\Id{{\rm Id}}

\usepackage{enumerate}

\newcommand{\func}[1]{\operatorname{#1}}

\begin{document}

\title*{Specker Algebras: A Survey}
\author{G.~Bezhanishvili, P.J.~Morandi, B.~Olberding}
\institute{ \at Department of Mathematical Sciences, New Mexico State University, Las Cruces NM 88003 \blu USA\blk,
\email{guram@nmsu.edu, pmorandi@nmsu.edu, \blu bruce\blk@nmsu.edu}}
%
%
\maketitle
\vspace{-1.25in}
\begin{center}
\emph{Dedicated to the memory of Elbert Walker}
\end{center}

\bigskip

\abstract*{For a commutative ring $R$ with identity, a Specker $R$-algebra is a commutative unital $R$-algebra generated by a Boolean algebra of idempotents,
each nonzero element of which is faithful. Such algebras have arisen in the study of $\ell$-groups, idempotent-generated rings,
Boolean powers of commutative rings, Pierce duality, and rings of continuous real-valued functions. We trace the origin of this notion from
early studies of subgroups of bounded integer-valued functions to a variety of current contexts involving ring-theoretic, topological,
and homological aspects of idempotent-generated algebras.}

\abstract{For a commutative ring $R$ with identity, a Specker $R$-algebra is a commutative unital $R$-algebra generated by a Boolean algebra of idempotents,
each nonzero element of which is faithful. Such algebras have arisen in the study of $\ell$-groups, idempotent-generated rings,
Boolean powers of commutative rings, Pierce duality, and rings of continuous real-valued functions. We trace the origin of this notion from
early studies of subgroups of bounded integer-valued functions to a variety of current contexts involving ring-theoretic, topological,
and homological aspects of idempotent-generated algebras.
}

\section{Introduction}
\label{intro}

Our goal in this article is to give some historical and practical motivation for a class of commutative algebras known as Specker algebras that
are encountered in various guises in contexts such as Baer-Specker theory, hyperarchimedean $\ell$-groups and vector lattices, rings of continuous
functions, Boolean powers of commutative rings, and Pierce sheaves.

Before outlining our survey of this topic, we wish to acknowledge first our debt to Elbert Walker as friend, mentor, colleague, and researcher.
While our topic does not directly build on Elbert's work, it cuts across several of his interests in a path that would be perfectly familiar
to him, from abelian group theory and ring theory to foundations, category theory, and homological algebra. When we arrived at NMSU at different
points during the years 1990-2002, Elbert was beginning his second research career in statistics and fuzzy logic. Characteristically, rather
than distancing him from the foundations and algebra research groups at NMSU, his new research only deepened his impact on these groups,
resulting in new collaborations among the seminar members, as well as new  directions for research that incorporated the algebraic approaches
to fuzzy logic   pioneered by Elbert and Carol Walker.
In fact, it was not unusual to hear Elbert speak in the algebra seminar one year on abelian $p$-groups, the next on type 2 fuzzy sets, and the
next on localization of categories.

It is in this same spirit of appreciation of multiple perspectives that we hope to present the topics in this article. Our approach to the topic
draws heavily on Elbert's influence and persuasive viewpoint on mathematics. As Elbert would often remind us, his interest was not only in the
proof of a theorem but {\it why} the theorem was true. The topic of this survey grew from just such a desire to understand a similar ``why''
for certain algebraic-topological dualities such as Gelfand-Naimark-Stone duality, Kakutani-Krein-Yosida duality,
de Vries duality, and Pierce duality. While these dualities remain far in the background of this article, connections to them can be seen
more explicitly  in \cite{BMMO15b, BMO13a, BMO16, BMO18} and will be further elaborated on in future articles. For the present article, we want instead to
focus on a rather simply defined class of commutative algebras, that of Specker algebras, which plays a key role in our algebraic approach
to these dualities.

In Section 2 we discuss the \blu Specker group concept\blk, which originates with Specker's work on subgroups of the Baer-Specker group.
We trace this concept through N\"obeling's results on freeness  and into Conrad's work, which gives an axiomatic description of Specker groups 
by locating them within the class of \blu $\ell$-groups (lattice-ordered groups). The culmination of this line of research is the representation of a Specker $\ell$-group as a group of $\mathbb{Z}$-valued continuous functions on a locally compact Hausdorff zero-dimensional space (see Theorem~\ref{rep}). Although we make no claim of originality for the results themselves, we give self-contained proofs of most of the results in Section 2, emphasizing a duality-theoretic approach in the spirit of Stone. \blk



Using the fact from Section 2 that multiplication is always present on Specker $\ell$-groups,  we shift focus in Section 3 from groups to rings,
and from rings to algebras. We define Specker $R$-algebras for an arbitrary commutative ring with identity, and we discuss several of their
properties. In this treatment, we follow \cite{BMMO15a}, but we also point out additional connections and applications from \cite{BMMO15b, BMO18}.

In the final section of the paper, we consider Specker $R$-algebras in the two classical cases, $R = {\mathbb{Z}}$ and $R  = {\mathbb{R}}$.
We use the former case and the work of Bergman \cite{Ber72}  to revisit the freeness results for Specker $\ell$-groups discussed in
Section 2. In the case $R = {\mathbb{R}}$, we discuss the role Specker algebras play in the theory of continuous real-valued functions
on Stone spaces.

\section{Specker $\ell$-groups}
\label{BS}

The origins of Specker $\ell$-groups can be traced back to Baer-Specker theory. The product $G= \prod_{i=1}^\infty {\mathbb{Z}}$ of countably many copies of
$\mathbb Z$ is termed the {\it Baer-Specker group}, so called because Baer \cite{Bae37} proved first that this group is not free while
Specker \cite{Spe50} proved that certain subgroups of $G$, including the countable subgroups, are free. N\"{o}beling \cite{Nob68} generalized
some of Specker's work to obtain that for any set $X$ the additive group $B(X,{\mathbb{Z}})$ of bounded functions from $X$ to ${\mathbb{Z}}$
is free and has a basis of characteristic functions. As we discuss in Section~\ref{sec:Bergman}, Bergman \cite{Ber72} has given a different proof
of this fact that is important for our point of view. N\"{o}beling then considered certain subgroups of $B(X,{\mathbb{Z}})$:

\begin{definition} {(N\"{o}beling \cite[p.~41]{Nob68})}
A group $G$ is called a {\it Specker group} if \blu $G$ is a subgroup \blk of $B(X,{\mathbb{Z}})$ for some set $X$, and for each function
$g \in G$ and $n \in {\mathbb{Z}}$, the characteristic function of $g^{-1}(n)$ is in $G$.
\end{definition}

A theorem of N\"{o}beling \cite[Thm.~2]{Nob68} implies that every Specker group is free and has a basis of characteristic functions.
Thus, N\"{o}beling's work gives a new class of subgroups of ${\mathbb{Z}}^X$ which are free. However, N\"{o}beling's definition is lacking
the robustness of an axiomatic definition in the sense that it depends on an embedding into a power of ${\mathbb{Z}}$.

This lack of an axiomatization is addressed by Conrad in the 1974 article \cite{Con74}. He first observes  that Specker groups are necessarily \blu $\ell$-groups\blk\footnote{For background on $\ell$-groups, see \cite[Ch.~XIII]{Bir79}.} with respect to the pointwise order. In fact, the Specker subgroups of $B(X,{\mathbb{Z}})$ for some
set $X$ are precisely the \blu $\ell$-subgroups \blk of $B(X,{\mathbb{Z}})$ \cite[4.2]{Con74}. Moreover, every Specker group is closed with respect
to the pointwise multiplication of  $B(X,{\mathbb{Z}})$ \cite[4.2]{Con74}. As with the definition of a Specker group, both of these properties
assume an ambient group $B(X,{\mathbb{Z}})$ in order to formulate them. To free himself of this ambient group, Conrad introduces the following
notions.

\begin{definition}
Let $G$ be an abelian $\ell$-group.
\begin{enumerate}[$(1)$]
\item An element $s \in G$ is {\it singular} if $s \ge 0$ and $g \wedge (s-g) = 0$ for each $0 \leq g \le s$.
\item Call $G$ a {\em Specker $\ell$-group} if $G$ is generated by its singular elements.
\end{enumerate}
\end{definition}

\begin{remark}
\begin{enumerate}[$(1)$]
\item[]
\item The above definition of a singular element is slightly different from Conrad's definition \cite[4.3]{Con74} in that he
does not assume 0 to be singular.
\item Conrad \cite[4.6]{Con74} originally called Specker $\ell$-groups $S$-groups.
\end{enumerate}
\end{remark}

\begin{theorem} \label{Conrad}  {\em (Conrad \cite[4.2]{Con74})}
An abelian group $G$ is isomorphic to a Specker group iff there is a lattice order on $G$ such that $G$ is a Specker $\ell$-group.
\end{theorem}

Theorem~\ref{Conrad} provides an axiomatization of Specker groups that is independent of any ambient embedding.
A detailed study of Specker $\ell$-groups was conducted by Conrad \cite{Con74} and Conrad and Darnel \cite{CD92,CD98,CD01}.
One of the key results, which can be found in \cite[Prop.~57.21]{Dar95}, is that an $\ell$-group $G$ is a Specker $\ell$-group iff there is a locally compact Hausdorff zero-dimensional space
$X$ such that $G$ is isomorphic to the $\ell$-group $C_k(X,\mathbb{Z})$ of all continuous $\mathbb{Z}$-valued functions with compact support. We give a different proof of this result relying more explicitly on Boolean algebra theoretic methods.

We recall that a {\em generalized Boolean algebra} is a distributive lattice $L$ with $0$ such that for each $a\in L$, the interval $[0,a]$ is a
Boolean algebra. Clearly if a generalized Boolean algebra has a $1$, then it is a Boolean algebra.
By \cite[Prop.~2.5]{CD98}, the set of singular elements of an abelian $\ell$-group forms a generalized Boolean algebra. We give an elementary proof for the reader's convenience.

\begin{lemma} \label{singular unit}
Let $G$ be an abelian $\ell$-group and let $S$ be the set of all singular elements of $G$.
\begin{enumerate}[$(1)$]
\item If $g$ is singular and $0 \le f \le g$, then $f$ is singular.
\item If $g_1, g_2$ are singular, then $g_1 \vee g_2$ is singular.
\item $S$ is a generalized Boolean algebra under the meet and join operations of $G$.
\end{enumerate}
\end{lemma}

\begin{proof}
(1). Suppose that $g$ is singular and $0 \le f \le g$. Let $e \in G$ with $0 \le e \le f$. Since $f \le g$ and $g$ is singular,
$0 \le (f-e) \wedge e \le (g-e) \wedge e = 0$ . Therefore, $(f-e) \wedge e = 0$, so $f$ is singular.

(2). Let $f \in G$ with $0 \le f \le g_1 \vee g_2$. Then
\[
f \wedge [(g_1 \vee g_2) - f] = f \wedge [(g_1 - f) \vee (g_2 - f)] = [f \wedge (g_1 - f)] \vee [f \wedge (g_2 - f)].
\]
Now, $f \wedge (g_1 - f) \le f \wedge [g_1 - (f \wedge g_1)]$, and this latter term is bounded by $g_1$ as $g_1 - (f\wedge g_1) \le g_1$.
Consequently,\blu
\[
f \wedge [g_1 - (f \wedge g_1)] = f\wedge g_1 \wedge [g_1 - (f \wedge g_1)]. 
\]\blk
Since $g_1$ is singular and
$0 \le f\wedge g_1 \le g_1$, we have $(f \wedge g_1) \wedge [g_1 - (f\wedge g_1)] = 0$. Therefore, $f \wedge (g_1-f) \le 0$. Similarly,
$f \wedge (g_2 - f) \le 0$, and so from the above equation, $f \wedge [(g_1 \vee g_2) - f] \le 0$. Clearly this is nonnegative, which
yields $f \wedge [(g_1 \vee g_2) - f] = 0$. Thus, $g_1 \vee g_2$ is singular.

(3). By (1) and (2), $S$ is closed under meet and join. Clearly $0 \in S$. Therefore, $S$ is a sublattice of $G$ with bottom.
Since $G$ is a distributive lattice, so is $S$. Let $g \in S$ and $0 \le f \le g$.
Because $g$ is singular, $f \wedge (g-f) = 0$. From this, by the $\ell$-group identity $x + y = x \vee y + x \wedge y$, we have $f \vee (g-f) = f + (g-f) = g$.
Thus, $g-f$ is the complement of $f$ in the interval $[0,g] \subseteq S$, proving that $S$ is a generalized Boolean algebra.
\qed
\end{proof}

Let $G$ be an abelian $\ell$-group. For $a\in G$, we recall that the \emph{positive} and \emph{negative} parts of $a$ are defined as
$a^+ = a \vee 0$ and $a^- = \blu -(a \wedge 0) = \blk (-a) \vee 0$, and we have $a = a^+ - a^-$. \blu In defining $a^-$ we follow \cite{LZ71} rather than \cite{Bir79}. \blk Also, the \emph{absolute value} of $a$ is defined as
$|a| = a \vee (-a)$, and we have $|a| = a^+ + a^-$.
We call $a,b\in G$ {\em orthogonal} if $a\wedge b=0$; and a representation
$a = \sum_{i=1}^n m_i g_i$ {\em orthogonal} if the $g_i$ are pairwise orthogonal. The next lemma provides a useful tool for working with Specker $\ell$-groups; for Part (1) see \cite[Prop.~1.2]{CD98}. 

\begin{lemma} \label{orthogonality}
Let $G$ be a Specker $\ell$-group and let $S$ be the set of all singular elements of $G$.
\begin{enumerate}[$(1)$]
\item Every element $a \in G$ has an orthogonal representation $a = \sum_{i=1}^n m_i g_i$ with the $m_i \in \mathbb{Z}$ and $g_i \in S$.
\item If $a = \sum_{i=1}^n m_i g_i$ is an orthogonal representation with each $0 \le m_i \in \mathbb{Z}$ and $g_i \in S$, then $a = \bigvee_{i=1}^n m_i g_i$.
\item If $a,b \in G$, then there are compatible orthogonal representations $a = \sum_{i=1}^n m_i g_i$ and $b = \sum_{i=1}^n p_i g_i$
for some $m_i, p_i \in \mathbb{Z}$ and $g_i\in S$.
\item Let $a = \sum_{i=1}^n m_ig_i$ be an orthogonal representation with each $0< g_i\in S$. Then $a^+ = \sum\{m_ig_i \mid m_i > 0\}$ and
$a^- = \sum\{ -m_ig_i \mid m_i < 0\}$.
\item If $a \in G$ has an orthogonal representation $a = \sum_{i=1}^n m_ig_i$ with each $0 < g_i \in S$,
then $a \ge 0$ iff each $m_i \ge 0$.
\end{enumerate}
\end{lemma}

\begin{proof}
(1). Let $g,h \in S$. Then $g_1 := g - (g\wedge h)$, $h_1 := h - (g\wedge h)$, and $g\wedge h$ are all singular by Lemma~\ref{singular unit}(1).
We have\blu
\[
g_1 \wedge (g\wedge h) = [g- (g \wedge h)]\wedge (g\wedge h) = 0
\]\blk
since $g$ is singular. Similarly, $h_1 \wedge (g\wedge h) = 0$.
Also,
\[
g_1 \wedge h_1 = [g - (g \wedge h)] \wedge [h - (g \wedge h)] = (g \wedge h) - (g \wedge h) = 0.
\]
Therefore, $g_1$ and $h_1$ are orthogonal. Thus, if $a = ng + mh$ for $n,m \in \mathbb{Z}$, then $a = ng_1 + mh_1 + (n+m)(g\wedge h)$, which
is an orthogonal representation. A simple induction argument then shows that any $\mathbb{Z}$-linear combination of singular
elements can be rewritten into an orthogonal representation.

(2). Suppose that $a = m_1 g_1 + m_2 g_2$ with $g_1 \wedge g_2 = 0$ and $0 \le m_1, m_2$. Without loss of generality, we may assume that
$m_1 \le m_2$. Then
\[
0 \le m_1 g_1 \wedge m_2 g_2 \le m_2 g_1 \wedge m_2g_2 = m_2(g_1 \wedge g_2) = 0.
\]
Thus, $m_1 g_1 + m_2 g_2 = m_1g_1 \vee m_2g_2$. A simple induction argument then yields the result.

(3). Let $a, b \in G$. By (1), $a = \sum_{i=1}^n m_i g_i$ and $b = \sum_{j=1}^m p_j h_j$ with the $m_i \in \mathbb{Z}$ (resp.~$p_j \in \mathbb{Z}$)
and the $g_i \in S$ (resp.~$h_j\in S$) pairwise orthogonal. Set $g = \bigvee_{i=1}^n g_i$ and $h = \bigvee_{j=1}^m h_j$,
and let $u = g \vee h$. By \blu Lemma~\ref{singular unit}\blk, the interval $[0,u]$ is a Boolean algebra. Let $g'$ (resp.~$h'$) be the complement
of $g$ (resp.~$h$) in $[0,u]$. By adding $0g'$ (resp.~$0h'$) to the orthogonal representations of $a, b$, we may assume that
$g = h = u$. From this we see that $g_i = g_i \wedge (\bigvee_{j=1}^m h_j) = \bigvee_{j=1}^m (g_i \wedge h_j)$. By (2),
$g_i = \sum_{j=1}^m (g_i \wedge h_j)$ and similarly $h_j = \sum_{i=1}^n (g_i \wedge h_j)$. Consequently, $a = \sum_i m_i g_i = \sum_{i,j} m_i (g_i \wedge h_j)$ and
$b = \sum_{i,j} p_j (g_i \wedge h_j)$ are compatible orthogonal representations.

(4). Let $a = \sum_{i=1}^n m_i g_i$ with the $0 \ne g_i\in S$ pairwise orthogonal. Set $b = \sum\{ m_i g_i \mid m_i > 0\}$ and
$c = \sum \{ -m_i g_i \mid m_i < 0 \}$. Then $a = b-c$ and $b,c \ge 0$. Moreover, by (2), $b = \bigvee \{ m_ig_i \mid m_i > 0\}$
and $c = \bigvee \{ -m_ig_i \mid m_i < 0\}$. Because the $g_i$ are pairwise orthogonal, we see that $b \wedge c = 0$.
Since $b, c \ge 0$ and $a = b-c$, we conclude that $b = a^+$ and $c = a^-$ (see \cite[p.~295, Lem.~4]{Bir79}).

(5). Let $a = \sum_{i=1}^n m_i g_i$ with the $0 \ne g_i\in S$ pairwise orthogonal. Since each $g_i \ge 0$, it is clear that if each
$m_i \ge 0$, then $a \ge 0$. Conversely, suppose that $a \ge 0$. Then $a^- = 0$. Since each $g_i > 0$, if some $m_i < 0$,
then $a^- > 0$ by (4). This contradiction implies that all $m_i \ge 0$.
\qed
\end{proof}

We next show that each nonzero element of a Specker $\ell$-group has a unique orthogonal representation.
This is stated without proof in \cite{CD98}. 

\begin{lemma} \label{unique}
Let $G$ be a Specker $\ell$-group and let $S$ be the set of all singular elements of $G$.
\begin{enumerate}[$(1)$]
\item If $g,h \in S$ are nonzero and $n,m$ are positive integers, then $ng \le mh$ iff $g \le h$ and $n \le m$.
\item Every nonzero element of $G$ has a unique orthogonal representation $\sum_{i=1}^n m_ig_i$ with the $m_i$ distinct and nonzero
and the $g_i\in S$ nonzero.
\end{enumerate}
\end{lemma}

\begin{proof}
(1). One direction is clear. For the other, suppose that $ng \le mh$. Set $g_1 = g - (g\wedge h)$. Since $g$ is singular, $g_1 \wedge (g\wedge h) \blu = (g - g \wedge h) \wedge (g \wedge h) = 0$\blk, and since $g_1 \le g$, we have $g_1 \wedge h = 0$.
Because $g \le ng \le mh$, we see that
\[
0 \le g_1 \wedge mh \le mg_1 \wedge mh = m(g_1 \wedge h) = 0,
\]
which shows $g_1 \wedge mh = 0$. Thus, since
$g_1 \le mh$, we get $g_1 = 0$. Hence, $g = g\wedge h$, so $g \le h$. Suppose $n > m$. We have $mg + (n-m) g \le mh$, so $(n-m) g \le m(h-g)$.
From $g \le h$ and $h$ singular it follows that $g \wedge (h-g) = 0$. Therefore, \blu\cite[p.~294, Thm.~5]{Bir79} implies $g \wedge m(h-g) = 0$, which further implies that \blk $(n-m)g \wedge m(h-g) = 0$ since $n-m \ge 0$.
This forces $(n-m)g = 0$. Because $G$ is torsion-free (see, e.g., \cite[p.~294, Cor.~1]{Bir79}) and $g$ is nonzero, $n=m$.
This contradiction shows $n \le m$.

(2). Let $0 \ne a \in G$. By Lemma~\ref{orthogonality}(1), we may write $a = \sum_{i=1}^n m_ig_i$ with the $m_i\in\mathbb Z$ and the $g_i\in S$
pairwise orthogonal. We may assume that the $m_i$ and $g_i$ are nonzero.
Let $\{p_1, \dots, p_t\}$ be the distinct elements of $\{m_1, \dots, m_n\}$, and let $k_i = \bigvee \{ g_j \mid m_j = p_i\}$. We have $\sum \{m_j g_j \mid m_j = p_i \} = p_i \sum \{g_j \mid m_j = p_i\}$.  Since the $g_i$ are \blu pairwise orthogonal\blk, by Lemma~\ref{orthogonality}(2), $\sum_{j}g_j = \bigvee_j g_j = k_i$.
Thus, $a = \sum_{i=1}^t p_i k_i$ is an orthogonal representation with distinct coefficients
and the \blu $k_i$ singular \blk and nonzero.

For uniqueness, suppose that $a \in G$ can be written as $a = \sum_{i=1}^r m_ig_i = \sum_{j=1}^s n_j h_j$ with the $m_i\in\mathbb Z$ (resp.~$n_j\in\mathbb Z$)
distinct and the $g_i\in S$ (resp.~$h_j\in S$) nonzero pairwise orthogonal. By setting the positive and negative parts equal and using
Lemma~\ref{orthogonality}(4), it suffices to assume all coefficients are positive. First suppose that $r=1$. We write $mg = \sum_{j=1}^s n_j h_j$.
Relabel if necessary to assume $n_1 = \max\{n_j\}$. Then
\[
mg \le n_1(h_1 + \cdots + h_s) = n_1 (h_1 \vee \cdots \vee h_s).
\]
By (1) we get $m \le n_1$ and $g \le h_1 \vee \cdots \vee h_s$. On the other hand, since all coefficients are positive,
$mg \ge n_1h_1$, so $m \ge n_1$ by (1). Therefore, $m = n_1$ and $g \ge h_1$. Thus, $m(g-h_1) = n_2h_2 + \cdots + n_sh_s$.
If $s \ge 2$, then repeating this argument shows $m = \max\{n_2,\dots, n_s\}$. But this is impossible since $m = n_1$,
which is strictly greater than all other $n_j$. This contradiction shows $s = 1$. Then $m = n_1$ and $g = h_1$.

Next, suppose $a = \sum_{i=1}^r m_ig_i = \sum_{j=1}^s n_j h_j$ with $r \ge 2$. Relabel if necessary to assume $m_1 = \max\{m_i\}$ and
$n_1 = \max\{n_j\}$. Then
\[
m_1g_1 \le n_1(h_1 + \cdots + h_s) = n_1(h_1 \vee \cdots \vee h_s).
\]
Therefore, by (1) again,
$m_1 \le n_1$. Reversing the roles of the $g_i$ and $h_j$ yields $n_1 \le m_1$. Thus, $m_1 = n_1$. Set
$g' = g_1 - (g_1 \wedge h_1)$ and $h' = h_1 - (g_1 \wedge h_1)$. Then
\[
m_1g' + m_2g_2 + \cdots + m_rg_r = m_1h' + n_2h_2 + \cdots + n_sh_s,
\]
so $m_1g' \le m_1h' + n_2h_2 + \cdots + n_sh_s$. Since $h',h_2,\dots,h_s$ are pairwise orthogonal and the $n_i$ are positive,
by Lemma~\ref{orthogonality}(2),
\[
m_1h' + n_2h_2 + \cdots + n_sh_s = m_1h' \vee n_2h_2 \vee \cdots \vee n_sh_s.
\]
Because $m_1g' \wedge m_1h' = m_1(g' \wedge h') =  0$, it follows that
$m_1g' \le n_2 h_2 + \cdots + n_sh_s$. If $g' > 0$, then repeating part of the $r=1$ argument shows $m_1 \le \max\{n_2,\dots, n_s\}$,
which is false. Therefore, $g' = 0$, so $g \le h$. Replacing $g'$ by $h'$ yields $h \le g$, so $g = h$ and $n_1 = m_1$. Thus,
$m_2g_2 + \cdots + m_rg_r = n_2h_2 + \cdots + n_sh_s$. The result then follows by induction.
\qed
\end{proof}

We are ready to give our proof of the following representation of Specker $\ell$-groups.

\begin{theorem} \label{rep}
Let $G$ be an abelian $\ell$-group. Then $G$ is a Specker $\ell$-group iff there is a locally compact Hausdorff zero-dimensional space $X$ such that
$G$ is isomorphic to the $\ell$-group $C_k(X, \mathbb{Z})$ of all continuous $\mathbb{Z}$-valued functions on $X$ with compact support.
\end{theorem}

\begin{proof}
We first show that the $\ell$-group $C_k(X, \mathbb{Z})$ of all continuous $\mathbb{Z}$-valued functions with compact support on a locally compact
Hausdorff zero-dimensional space is a Specker $\ell$-group. Let $f \in C_k(X, \mathbb{Z})$. \blu Since $f$ has compact support, $f(X)$ is a compact subset \blk of $\mathbb{Z}$, hence it is finite.
\blu Since  $f$ is \blk continuous, for each nonzero $m_i \in f(X)$, we have $f^{-1}(m_i)$ is a compact clopen subset of $X$.
From this it follows that $f = \sum_{i=1}^n m_i \chi_{U_i}$ for some nonzero integers $m_i$ and
nonempty compact clopen subsets $U_i$ of $X$.
It remains to show, for each compact clopen $U$, that $\chi_U$ is singular.
Suppose that $0 \le g \le \chi_U$. Since $g$ is $\mathbb{Z}$-valued, there is $V \subseteq U$ with $g = \chi_V$.
Then $\chi_U - g = \chi_{U \setminus V}$, so
$g \wedge (\chi_U - g) = \chi_V \wedge \chi_{U \setminus V} = \chi_{U \cap (U \setminus V)} = \chi_\varnothing = 0$.
Thus, $\chi_U$ is singular.

For the converse, by Lemma~\ref{singular unit}(3), the set $S$ of singular elements of $G$ is a generalized Boolean algebra. Let $X$ be the Stone dual of $S$.
Then $X$ is a locally compact Hausdorff zero-dimensional space \cite[Sec.~I.2]{Sto37b}. We define a map $\alpha : G \to C_k(X, \mathbb{Z})$ as follows.
If $g \in S$, let $K_g$ be the corresponding compact clopen subset of $X$. Then the characteristic function $\chi_{K_g}$
is a continuous $\mathbb{Z}$-valued function on $X$ with compact support, and so $\chi_{K_g} \in C_k(X, \mathbb{Z})$.
If $0 \ne a \in G$, by Lemma~\ref{unique}(2), let $a = m_1g_1 + \cdots + m_ng_n$ be the unique orthogonal representation of $a$, and set
$\alpha(a) = m_1\chi_{K_{g_1}} + \cdots + m_n\chi_{K_{g_n}}$. We point out that $\alpha(a)$ can be computed from any orthogonal representation
$a = m_1g_1 + \cdots + m_ng_n$.
For, if any of $m_i\in\mathbb Z$ or $g_i\in S$ is zero, then $m_ig_i=0=\alpha(m_ig_i)$. Otherwise,
the proof of Lemma~\ref{unique}(2) shows that we obtain the
unique representation of Lemma~\ref{unique}(2) by replacing a piece of the sum of the form $mg_{i_1} + \cdots + mg_{i_t}$ by $mg$,
where $g = g_{i_1} \vee \cdots \vee g_{i_t}$. But $K_g = K_{g_{i_1}} \cup \cdots \cup K_{g_{i_t}}$ by Stone duality.
Thus, $\alpha(a) = m_1\chi_{K_{g_1}} + \cdots + m_n \chi_{K_{g_n}}$.

Let $a,b \in G$. By Lemma~\ref{orthogonality}(3), we may write $a = \sum_{i=1}^n m_i g_i$ and
$b = \sum_{i=1}^n p_i g_i$ for some pairwise orthogonal set of singular elements. Since $a+b = \sum_{i=1}^n (m_i + p_i) g_i$, it then
follows that $\alpha(a+b) = \alpha(a) + \alpha(b)$, and so $\alpha$ is a group homomorphism. Furthermore, if $\alpha(a) = 0$,
then the formula for $\alpha(a)$ shows that all coefficients are $0$, and so $a=0$. To see that $\alpha$ is onto, let
$f \in C_k(X, \mathbb{Z})$. Then
$f$ is finitely valued. Let $m_1, \dots, m_n$ be the distinct nonzero values of $f$. Since $f$ has compact support,
each $K_i = f^{-1}(m_i)$ is compact clopen in $X$,
and it is easy to see that $f = m_1\chi_{K_1} + \cdots + m_n\chi_{K_n}$. By Stone duality, for each $i$ there is $g_i \in S$ with
$K_i = K_{g_i}$, so $f = \alpha(m_1g_1 + \cdots + m_ng_n)$. Therefore, $\alpha$ is a group isomorphism. Since $G$ and $C_k(X, \mathbb{Z})$ are
Specker $\ell$-groups, it follows from  Lemma~\ref{orthogonality}(5) that $a\ge 0$ iff $\alpha(a)\ge 0$. Thus, $\alpha$ is an order isomorphism, hence an isomorphism of $\ell$-groups.\qed
\end{proof}

\begin{remark}
Since $\mathbb Z$ is archimedean, it is an immediate consequence of Theorem~\ref{rep} that every Specker $\ell$-group is archimedean.
\end{remark}

\begin{corollary} \label{mult} {\em (Conrad \cite[4.7]{Con74})}
Let $G$ be a Specker $\ell$-group. Then there is a unique multiplication on $G$ which makes $G$ into a commutative ring such that
$gh = g \wedge h$ for all \blu singular elements $g,h$. \blk Consequently, singular elements are precisely the idempotents of $G$.
\end{corollary}

\begin{proof}
By Theorem~\ref{rep}, we may identify $G$ with $C_k(X, \mathbb{Z})$ for some locally compact Hausdorff zero-dimensional space $X$.
Pointwise multiplication makes $C_k(X, \mathbb{Z})$ into a commutative ring.
Since each $g \in S$ is identified with the characteristic function $\chi_{K_g}$,
if $g,h \in S$, then $\chi_{K_g} \wedge \chi_{K_h} = \chi_{K_g \cap K_h} = \chi_{K_g}\cdot \chi_{K_h}$,
and hence $gh = g \wedge h$. As idempotents of $C_k(X, \mathbb{Z})$ are characteristic functions of compact clopens,
we conclude that singular elements are precisely the idempotents.
Finally, since $G$ is generated by singular elements, the equation $gh = g \wedge h$ gives that the multiplication is unique.
\qed
\end{proof}

Let $G$ be an abelian $\ell$-group. We recall that a positive element $u\in G$ is a {\em weak order-unit} if for any $a\ge 0$,
from $a\wedge u=0$ it follows that $a=0$; and that $u$ is a {\em strong order-unit} if for each $a\in G$ there is $n\in\mathbb N$
such that $a\le nu$. The concept of a strong order-unit is in general stronger than that of a weak order-unit (see, e.g., 
\blu \cite[p.~308]{Bir79}). \blk However, the two concepts are equivalent for Specker $\ell$-groups. While this can be derived from \cite[Thm.~55.1]{Dar95},  we give a direct proof in the next lemma.

\begin{lemma} \label{BA}
Let $G$ be an abelian $\ell$-group and let $S$ be the set of singular elements of $G$. Suppose that $u \in G$ is a weak order-unit and 
$u = m_1g_1 + \cdots + m_ng_n$ is an orthogonal representation with the $m_i > 0$ and the $g_i \in S$. Let $g = g_1 \vee \cdots \vee g_n$.
\begin{enumerate}[$(1)$]
\item $g$ is the largest element of $S$ and is a strong order-unit.
\item $u$ is a strong order-unit.
\item $S$ is a Boolean algebra.
\end{enumerate}
\end{lemma}

\begin{proof}
(1). Clearly $g\in S$ by Lemma~\ref{singular unit}(2). Let $h \in S$. Applying the lemma again yields $g \vee h\in S$,
so $g \wedge [(g\vee h) - g] = 0$. Set $a = (g \vee h) - g$. Then 
\[
0\le g_i \wedge a = g_i \wedge [(g\vee h) - g] \le g \wedge [(g\vee h) - g] = 0.
\]
Therefore, $g_i \wedge a = 0$, so $m_i g_i \wedge a = 0$ for each $i$. Thus, by Lemma~\ref{orthogonality}(2), $u \wedge a = 0$. Since $u$ is 
a weak order-unit, $a = 0$. Consequently, \blu $g\vee h = g$, \blk so $h \le g$. This proves that $g$ is the largest element of $S$. To see that it 
is a strong order-unit, let $b \in G$ and let $b = p_1h_1 + \cdots + p_rh_r$ be an orthogonal representation with the $p_i\in\mathbb Z$ and
the $h_i\in S$. If $m = \sum_{i=1}^r |p_i|$, then $b \le mg$ since each $h_i \le g$. This shows that $g$ is a strong order-unit.

(2). This is immediate from (1) since $g \le u$ as each $m_i \ge 1$.

(3). This is immediate from (1) and Lemma~\ref{singular unit}(3).
\qed
\end{proof}

Another consequence of Theorem~\ref{rep} characterizes Specker $\ell$-groups with a weak order-unit. \blu We recall that a \emph{Stone space} is a compact Hausdorff zero-dimensional space.\blk

\begin{corollary} \label{rep cor}
Let $G$ be an abelian $\ell$-group. Then $G$ is a Specker $\ell$-group with a weak order-unit iff there is a Stone space $X$ such that $G$ is
isomorphic to the $\ell$-group $C(X, \mathbb{Z})$ of all continuous $\mathbb{Z}$-valued functions on $X$.
\end{corollary}

\begin{proof}
Suppose that $G$ is isomorphic to $C(X, \mathbb{Z})$ for a Stone space $X$. Since $X$ is compact, $C(X, \mathbb{Z}) = C_k(X, \mathbb{Z})$,
and so $G$ is a Specker $\ell$-group by Theorem~\ref{rep}. Moreover, the constant function 1 is a weak order-unit of $C(X, \mathbb{Z})$,
hence $G$ has a weak order-unit.
Conversely, if $G$ is a Specker $\ell$-group with a weak order-unit, then by Lemma~\ref{BA}(3), $S$ is a Boolean algebra, and
hence its dual $X$ is a Stone space. Therefore, since $X$ is compact, $C(X, \mathbb{Z}) = C_k(X, \mathbb{Z})$.
Applying Theorem~\ref{rep} again finishes the proof.
\qed
\end{proof}

\begin{remark}
Corollary~\ref{rep cor} suggests another interpretation of Specker $\ell$-groups having a weak order-unit.
Recall (see, e.g., \cite{BN80} or \cite[Sec.~IV.5]{BS81}) that if $B$ is a Boolean algebra and $A$ is an algebra of some type,
then the (bounded) \emph{Boolean power}
of $A$ by $B$ is the algebra $C(X,A)$ of the same type consisting of the continuous functions from the Stone space
$X$ of $B$ to the discrete space $A$.
If we take $A$ to be
$\mathbb{Z}$, then the Boolean power $C(X,\mathbb{Z})$ is a Specker $\ell$-group with a weak order-unit;
and by Corollary~\ref{rep cor}, the Specker $\ell$-groups with a weak order-unit can be reinterpreted
as the Boolean powers of ${\mathbb{Z}}$. \blu In Section 3 we \blk show that a similar result holds for our notion of a
Specker algebra, i.e., that the Specker $R$-algebras are precisely the Boolean powers of the commutative ring $R$.
\end{remark}

We conclude this section by showing that Specker $\ell$-groups with a weak order-unit carry all the information about Specker $\ell$-groups
in that each Specker $\ell$-group is an $\ell$-ideal of a Specker $\ell$-group with a weak order-unit (see \cite[p.~209]{Con74}). We recall that an $\ell$-ideal of
an $\ell$-group $G$ is a normal subgroup $N$ of $G$ satisfying that $a \in N$ and $|b| \le |a|$ imply $b \in N$.
It is well known (see, e.g., \cite[p.~304, Thm.~15]{Bir79}) that $\ell$-ideals are precisely the kernels of $\ell$-group homomorphisms.

\begin{theorem} \label{ideal}
Let $G$ be an abelian $\ell$-group. Then $G$ is a Specker $\ell$-group iff $G$ is isomorphic to an $\ell$-ideal in a Specker $\ell$-group
with a weak order-unit.
\end{theorem}

\begin{proof}
We prove one direction by showing that an $\ell$-ideal $N$ of a Specker $\ell$-group $H$ with a weak order-unit is itself a Specker $\ell$-group.
It is clear that $N$ is an abelian $\ell$-group. Let $a \in N$. First suppose that $0 \le a$, and write $a = m_1g_1 + \cdots + m_n g_n$
with the $m_i \in \mathbb Z$ and the $g_i\in S$ pairwise orthogonal. By Lemma~\ref{orthogonality}(5), $m_i \ge 0$, so
$0 \le g_i \le a$, and hence $g_i \in N$  for each $i$. \blu If $g \in N$ is singular \blk in $H$, then it is singular in $N$,
so $a$ is a sum of singular elements of $N$. For $a$ general, applying the previous argument to $a^+$ and $a^-$ shows that each is
a sum of singular elements in $N$. Consequently, $N$ is generated by its singular elements, and hence $N$ is a Specker $\ell$-group.

For the converse, if $G$ has a weak order-unit, there is nothing to show. Suppose $G$ does not have a weak order-unit.
By Theorem~\ref{rep}, there is a locally compact Hausdorff zero-dimensional space $X$ such that $G$ is isomorphic to $C_k(X, \mathbb{Z})$.
Since $G$ does not have a weak order-unit, $X$ is not compact.
Let $Y = X \cup \{\infty\}$ be the one-point compactification of $X$. Then $Y$ is a Stone space. By Corollary~\ref{rep cor},
$C(Y, \mathbb{Z})$ is a Specker $\ell$-group with a weak order-unit. We embed $C_k(X, \mathbb{Z})$ into $C(Y, \mathbb{Z})$
by extending each $f \in C_k(X, \mathbb{Z})$ by setting $f(\infty) = 0$. This extension to $Y$ is continuous since $f^{-1}(0)$
is then the complement of a compact subset of $X$, and so is an open neighborhood of $\infty$. Under this embedding $C_k(X, \mathbb{Z})$
is sent to the kernel of the $\ell$-group homomorphism $\alpha : C(Y, \mathbb{Z}) \to \mathbb{Z}$ defined by sending $f$ to $f(\infty)$.
Thus, $G$ is isomorphic to an $\ell$-ideal of $C(Y, \mathbb{Z})$.
\qed
\end{proof}

\begin{remark}
Since $\mathbb{Z}$ is simple as an $\ell$-group, the kernel of the $\ell$-group homomorphism $\alpha : C(Y, \mathbb{Z}) \to \mathbb{Z}$ in Theorem~\ref{ideal} is a maximal $\ell$-ideal. Thus, every Specker $\ell$-group is isomorphic to a maximal $\ell$-ideal in a Specker $\ell$-group with a weak order-unit.
\end{remark}

\section{Specker algebras}
\label{Sp}

Corollary~\ref{mult} induces the structure of a commutative ring on each Specker $\ell$-group $G$ under which singular elements
are precisely the idempotents of the ring. If in addition $G$ has a weak order-unit,
then by Corollary~\ref{rep cor}, $G$ can be viewed as a commutative ring with multiplicative identity 1.
With this in mind, in this section we shift our focus from Specker $\ell$-groups to commutative rings,
since when dealing with Specker $\ell$-groups, the ring-theoretic structure is always present. \blu We leave most of the proofs out and refer the reader to \cite{BMMO15a} for details.\blk

\bigskip\noindent\textbf{Convention}. From now on all rings and algebras we consider will be assumed to be commutative, and have a multiplicative identity unless otherwise specified.\bigskip

Every Specker $\ell$-group
can be viewed as a torsion-free ${\mathbb{Z}}$-algebra \blu (possibly without 1) \blk generated by its idempotents.
This is the motivation for the notion of a Specker algebra over a ring.
However, because we wish to have a notion of Specker algebra robust enough to include cases where the base ring can be any
ring, the dependence
on the condition that the algebra be torsion-free over its base ring is problematic for rings with zero-divisors.
This can be circumvented by a closer analysis of what suffices for an idempotent-generated ${\mathbb{Z}}$-algebra $A$ to be torsion-free:
Let  $0 \ne a \in A$ and
let $a = \sum_{i =1}^n m_i e_i$ be an orthogonal representation with each $0 \ne m_i \in {\mathbb{Z}}$ and the idempotents $e_i$ nonzero.
If $k \in {\mathbb{Z}}$, then the fact that the $e_i$ are pairwise orthogonal implies that  $ka = 0$ iff $km_i e_i = 0$ for each $i$.
Thus, $A$ is torsion-free as a ${\mathbb{Z}}$-module iff each nonzero idempotent in $A$ is {\it faithful},
meaning that no nonzero element of the base ring ${\mathbb{Z}}$ annihilates it.

This suggests that for a ring $R$, the correct notion of a Specker $R$-algebra is an idempotent-generated $R$-algebra
for which each nonzero idempotent is faithful. However, this once again presents a problem for the level of generality in which
we \blu wish to work: \blk If $R$ has a nontrivial idempotent $e \ne 0,1$, then since $(1-e)e = 0$, the idempotent $e$ is not faithful.
Thus, our provisional notion of a Specker algebra rules out any choices for the base ring in which there is a nontrivial idempotent.
In other words, such a definition is only useful for indecomposable rings.

For this reason, we propose in Definition~\ref{def of Specker} a more subtle way of capturing the essential properties of the algebras
arising from Specker $\ell$-groups. We no longer require that all nontrivial idempotents be faithful, since that proves much too
restrictive. We require instead that there are ``enough'' such faithful idempotents. Even the notion of ``enough'' is more subtle
than might be expected. To capture the full strength of the Specker condition as studied in Section 2, one needs not simply that
the algebra $A$ in question is generated by faithful idempotents. Rather, in order to have a reasonable decomposition theory of
its elements along the lines of Lemma~\ref{orthogonality}, the algebra needs that these faithful idempotents occur as the nonzero
elements of a Boolean subalgebra of the Boolean algebra $\Id(A)$ of idempotents of $A$. In such a case, we say that $A$ has a
{\it generating algebra} of faithful idempotents. Among other things, this condition is of technical importance because it permits
refinements and coarsenings among faithful representations of the elements in the algebra. With this condition in place, we arrive at
our main definition.

\begin{definition} \cite[Def.~2.3]{BMMO15a}\label{def of Specker}
Let $R$ be a ring. We call an $R$-algebra $A$ a \emph{Specker $R$-algebra} if $A$ is an $R$-algebra
that has a generating algebra of faithful idempotents.
\end{definition}

\begin{example}
Let $R$ be a domain and $A$ an idempotent-generated $R$-algebra. We claim that $A$ is a Specker $R$-algebra iff $A$ is
torsion-free as an $R$-module and, if so, $\Id(A)$ is the unique generating algebra of faithful idempotents. For the first statement, if
$A$ is torsion-free, then each nonzero element is faithful, so $\Id(A)$ is a faithful generating algebra, which implies that $A$ is Specker.
Conversely, suppose that each nonzero $e \in A$ is faithful. Let $0 \ne a \in A$. As with Specker $\ell$-groups, there is an orthogonal
representation $a = r_1e_1 + \cdots + r_ne_n$ for some $0 \ne r_i \in R$ and $0 \ne e_i \in \Id(A)$ (see \cite[Lem.~2.1]{BMMO15a}).
If $r \in R$ with $ra = 0$, then multiplying the equation $ra = 0$ by $e_i$ yields
\blu $rr_ie_i = 0$. \blk Since $e_i$ is faithful, $rr_i = 0$. Because $R$ is a domain and $r_i \ne 0$, we see $r = 0$.
Thus, $A$ is torsion-free. Finally, we show that $\Id(A)$ is the only generating algebra.
Suppose that $B \subseteq \Id(A)$ is a generating algebra for $A$. Let $e \in \Id(A)$ be nonzero. There is an orthogonal representation
$e = r_1b_1 + \cdots + r_nb_n$ for some $0\ne r_i \in R$ and $0\ne b_i \in B$. Therefore, $e = e^2 = r_1^2 b_1 + \cdots + r_n^2 b_n$, and so
$(r_1^2-r_1)b_1 + \cdots + (r_n^2 - r_n)b_n = 0$. Multiplying by $b_i$ yields $(r_i^2 - r_i)b_i = 0$, which implies $r_i^2 = r_i$ \blu since $b_i$ is faithful. \blk
Since $R$ is a domain, $r_i = 1$. Consequently, $e$ is the sum of the $b_i$, and so $e$ is the join of these $b_i$. Thus, $e \in B$.
This proves $B = \Id(A)$, and so $\Id(A)$ is the unique generating algebra of faithful idempotents for $A$.
\end{example}

\blu A natural \blk question arising from Definition~\ref{def of Specker} is whether  each Specker $R$-algebra
has a unique generating algebra of faitfhful idempotents. While
this is the case iff $R$ is indecomposable (see Theorem~\ref{ind}),
it is of note that any two generating algebras of faithful idempotents are isomorphic as Boolean algebras \cite[Thm.~3.5]{BMMO15a}.

This and many other properties of Specker $R$-algebras follow from a representation theorem for Specker $R$-algebras that has its roots in
the work of Bergman \cite{Ber72} and Rota \cite{Rot73}. The goal of the representation is to encode  the intuition that a Specker
$R$-algebra is comprised of two pieces of data, the base ring $R$ and a Boolean algebra $B$ that is (isomorphic to) \blu a  generating algebra \blk
of faithful idempotents. This is done via a polynomial construction which we outline next.

\begin{definition}\cite[Def.~2.4]{BMMO15a}\label{def:R[B]}
Let $B$ be a Boolean algebra. We denote by $R[B]$ the quotient ring $R[\{x_e \mid e\in B\}]/I_B$ of the polynomial ring over $R$ in
variables indexed by the elements of $B$ modulo the ideal $I_B$ generated by the following elements, as $e,f$ range over $B$:
\[
x_{e\wedge f} - x_e x_f, \ \ x_{e\vee f} - (x_e + x_f - x_e x_f), \ \ x_{\lnot e} - (1-x_e), \ \ x_0.
\]
\end{definition}

This construction does indeed produce a Specker $R$-algebra:

\begin{proposition} \emph{\cite[Lem.~2.6, Thm.~2.7]{BMMO15a}}
Let $R$ be a ring and let  $B$ be a Boolean algebra. Then $R[B]$ is a Specker $R$-algebra such that, in the notation of
Definition~\emph{\ref{def:R[B]}}, $\{x_e+I_B\mid e \in B\}$ is a generating algebra of faithful idempotents that is isomorphic as a Boolean algebra to $B$ \blu via the map $i_B$ sending $e \in B$ to $x_e + I_B$. \blk
\end{proposition}

In fact, every Specker $R$-algebra arises in this fashion, as we point out next. A useful way of seeing this is to note first that the
following universal mapping property is an easy consequence of the definition of $R[B]$.

\begin{lemma} \emph{\cite[Lem.~2.5]{BMMO15a}\label{UMP}}
Let $A$ be an $R$-algebra. If $B$ is a Boolean algebra and $\sigma : B \to \Id(A)$ is a Boolean homomorphism, then there
is a unique $R$-algebra homomorphism $\alpha : R[B] \to A$ satisfying $\alpha \circ i_B = \sigma$.
\end{lemma}

If $A$ is a Specker $R$-algebra and $B$ is a generating algebra of faithful idempotents, Lemma~\ref{UMP} implies there is a unique
$R$-algebra homomorphism $\alpha:R[B] \to A$ induced by the inclusion $B\to\Id(A)$.
The map $\alpha$ is onto because the idempotents in $B$ generate $A$ as an
$R$-algebra. That it is one-to-one is a consequence of the fact that each nonzero element of $A$ can be decomposed into an
$R$-linear combination of faithful idempotents from the generating algebra (see \cite[Lem.~2.1]{BMMO15a} for more details). Therefore,
every Specker $R$-algebra is isomorphic to an $R$-algebra of the form $R[B]$.

\begin{remark}
Each idempotent-generated $R$-algebra $A$ is isomorphic to a quotient of a Specker $R$-algebra. To see this, let $B = \Id(A)$. By Lemma~\ref{UMP}, the identity function $B \to B$ induces an $R$-algebra homomorphism $\alpha : R[B] \to A$, sending $x_b + I_B$ to $b$. This map is onto since $A$ is generated by $B$. Therefore, $A$ is isomorphic to $R[B]/\ker(\alpha)$.
\end{remark}

These ideas can be carried further, as is done in \cite{BMMO15a}, to show that the interpretation of Specker $\ell$-groups from Section 2
involving Boolean powers can be proved in our setting also. In particular, a Specker $R$-algebra can be viewed as the ring of continuous
functions from a Stone space $X$ to the ring $R$ with the discrete topology (which we denote by $R_{\func{disc}}$). Putting all this together,
we have:

\begin{theorem} \label{main} {\em \cite[Thm.~2.7]{BMMO15a}}
Let $A$ be an $R$-algebra. The following are equivalent.
\begin{enumerate}[$(1)$]
\item $A$ is a Specker $R$-algebra.
\item $A$ is isomorphic to $R[B]$ for some Boolean algebra $B$.
\item $A$ is isomorphic to $C(X,R_{\func{disc}})$ for some Stone space $X$.
\item $A$ is isomorphic to a Boolean power of $R$.
\end{enumerate}
\end{theorem}

\begin{remark}
It is possible to generalize Boolean powers of totally ordered domains so that instead of working with Stone spaces one works in the generality
of compact Hausdorff spaces. For this we need to generalize Stone duality for Boolean algebras. While there are a number of such generalizations,
probably the closest in spirit to Stone duality is de Vries duality for compact Hausdorff spaces \cite{deV62}. This requires adding a binary
relation of proximity to the signature of Boolean algebras. The resulting structures are known as de Vries algebras \cite{Bez10}. Boolean powers
are generalized to de Vries powers in \cite{BMMO15b}.  The main idea\blu , modulo many technical details, \blk is to replace the finitely valued continuous functions on a Stone space
by the finitely valued normal functions on a compact Hausdorff space. It is then possible to extend the proximity on $B$ to that on $R[B]$,
and prove that the resulting category of ``proximity Specker $R$-algebras" is dually equivalent to the category of compact Hausdorff spaces
\cite[Thm.~8.6]{BMMO15b}.
\end{remark}

As the nature of the construction of $R[B]$ and its universal mapping property  suggest, there is a a functor
$\mathcal{S} : \mathbf{BA} \to \mathbf{Sp}_R$ from the category ${\bf BA}$ of Boolean algebras and Boolean homomorphisms
to the category ${\bf Sp}_R$ of Specker $R$-algebras and unital $R$-algebra homomorphisms, defined on objects by
$\mathcal{S}(B) = R[B]$; and for maps, if $\sigma : B \to C$ is a Boolean homomorphism, then $\mathcal{S}(\sigma)$ is the unique
$R$-algebra homomorphism $R[B] \to R[C]$ extending $\sigma$.

There is also the obvious functor in the other direction, i.e., the functor $\mathcal{I} : \mathbf{Sp}_R \to \mathbf{BA}$ defined
on objects by $\mathcal{I}(A) = \Id(A)$; and for maps, if $\alpha : A \to A'$ is a unital $R$-algebra homomorphism, then
$\mathcal{I}(\alpha) = \alpha|_{\Id(A)}$.

These functors form an adjunction \cite[Lem.~3.7]{BMMO15a}, but not in general an equivalence since the full algebra of idempotents of a
Specker $R$-algebra may not be faithful. It is of interest to know when this adjunction is an equivalence of categories, since in
such a case the category of Specker $R$-algebras is simply the category of Boolean algebras in another guise. Not surprisingly,
the presence of nontrivial idempotents in $R$ is precisely the obstruction to equivalence:

\begin{theorem}\emph{\cite[Thm.~3.8]{BMMO15a}} \label{ind}
The following are equivalent for a ring $R$.
\begin{enumerate}[$(1)$]
\item $R$ is indecomposable.
\item The functors $\mathcal{I}$ and $\mathcal{S}$ yield an equivalence of $\mathbf{Sp}_R$ and $\mathbf{BA}$.
\item The Specker $R$-algebras are the idempotent-generated $R$-algebras for which
each nonzero idempotent is faithful.
\end{enumerate}
\end{theorem}

\begin{remark} \label{rem: locally Specker}
If $R$ is not indecomposable, then $\mathbf{Sp}_R$ and $\mathbf{BA}$ are no longer equivalent because the functor $\mathcal{I}$ does not take into account the Boolean action $i : \Id(R) \to \Id(A)$.  Let $R$ be an arbitrary ring. By adding an additional class of generators, $ex_1 - x_{i(e)}$ for $e \in \Id(R)$ to the ideal $I_B$ in the definition of $R[B]$, we obtain an algebra $R\langle B\rangle$ which we called a \emph{locally Specker algebra} in \cite{BMO18}. The name is motivated by the fact that if $M$ is a maximal ideal of $\Id(R)$, then $R\langle B\rangle/MR\langle B\rangle$ is a Specker $R/MR$-algebra, hence $R\langle B\rangle$ is ``locally'' a Specker algebra at $M$. By \cite[Thm.~4.4]{BMO18} the category of locally Specker $R$-algebras is equivalent to the category of Boolean algebras with an $\Id(R)$-action, thus providing a refinement of Theorem~\ref{ind}.
\end{remark}

\begin{remark} \label{rem: sheaves}
It is an immediate consequence of Theorem~\ref{ind} and Stone duality that for $R$ indecomposable, the category $\mathbf{Sp}_R$ is dually equivalent to the category $\textbf{Stone}$ of Stone spaces and continuous maps. If $R$ is not indecomposable, then this duality no longer holds, but can be refined by a duality between the category of locally Specker $R$-algebras and the category of Stone bundles \cite[Prop.~4.7(2)]{BMO18}.
\end{remark}

\begin{remark}
By Pierce duality \cite{Pie67}, each ring $R$ can be represented as the ring of global sections of a sheaf of indecomposable rings over the Stone space $X$ of $\Id(R)$. In terms of Pierce duality, all locally Specker $R$-algebras arise in the following way: If $\overline{R}$ is the Pierce sheaf of $R$ over $X$ and $f : Y \to X$ is a Stone bundle, then the pullback \blu $f^{-1}\left(\overline{R}\right)$ \blk of $\overline{R}$ is a sheaf of rings over $Y$, and its ring of global sections is a locally Specker $R$-algebra \cite[Thm.~3.3]{BMO18}.
\end{remark}

\section{Classical cases: Specker ${\mathbb{Z}}$-algebras and  Specker $ {\mathbb{R}}$-algebras}

In this section we focus on Specker $R$-algebras in the special cases $R = {\mathbb{Z}}$ and $R = {\mathbb{R}}$. In both of these cases,
Theorem~\ref{ind} implies that the categories ${\bf Sp}_R$ of Specker $R$-algebras and {\bf BA} of Boolean algebras are equivalent, and hence ${\bf Sp}_R$ is dually equivalent to \blu the category of Stone spaces. \blk So in both cases a Specker $R$-algebra is completely determined by its Boolean algebra of idempotents. Not surprisingly, this Boolean algebra is at the heart of the applications we discuss in this section.

As we indicate in Subsection~\ref{sec:Bergman},  the theory of Specker $\ell$-groups can be recast in the language of rings,
and in particular, the Specker $\ell$-groups with a weak order-unit are precisely the Specker ${\mathbb{Z}}$-algebras. Following Bergman
\cite{Ber72}, we use this insight to recover N\"{o}beling's freeness result discussed in Section 2. In Subsection 4.2, we work instead with
Specker ${\mathbb{R}}$-algebras and show that these can be viewed as the finitely valued continuous real-valued functions on a Stone space.
More generally, we discuss the role that Specker ${\mathbb{R}}$-algebras play in the context of rings of continuous real-valued functions.

\subsection{Specker $\mathbb{Z}$-algebras} \label{sec:Bergman}

With the aim of connecting the material on Specker $\ell$-groups to Specker ${\mathbb{Z}}$-algebras, we first give a characterization of
Specker $\ell$-groups in the language of rings.

\begin{theorem} \label{ig}
A group $G$ admits a lattice order making it a Specker $\ell$-group iff $G$ is isomorphic to the additive group of a torsion-free
$\mathbb{Z}$-algebra \blu \emph{(}possibly without 1\emph{)} \blk which is generated by idempotents.
\end{theorem}

\begin{proof}
Let $G$ be a Specker $\ell$-group. By Theorem~\ref{ideal}, $G$ is isomorphic to an $\ell$-ideal of a Specker $\ell$-group $G'$ with a weak
order-unit. By Corollary~\ref{rep cor}, $G'$ is isomorphic to $C(X, \mathbb{Z})$ for some Stone space $X$. Under pointwise multiplication,
$C(X, \mathbb{Z})$ is a $\mathbb{Z}$-algebra. It is torsion-free and generated by its idempotents, which are continuous characteristic
functions. Thus, $G$ embeds in $C(X, \mathbb{Z})$ and, as singular elements of $G$ are sent to idempotents of $C(X, \mathbb{Z})$, $G$ is
isomorphic to the additive group of a subalgebra
of $C(X, \mathbb{Z})$, which is a torsion-free $\mathbb{Z}$-algebra generated by idempotents.

Conversely, suppose $G$ is isomorphic to the additive group of a torsion-free $\mathbb{Z}$-algebra $A$ which is generated by
idempotents. If $A$ has a 1, then by \cite[Thm.~4.1]{BMMO15a}, $A$ is a Specker $\mathbb{Z}$-algebra, so by Theorem~\ref{main}, $A$ is
isomorphic to $C(Y, \mathbb{Z})$, where $Y$ is the Stone space of $\Id(A)$. Therefore, $G$ is a Specker $\ell$-group by Corollary~\ref{rep cor}.

If $A$ does not have a 1, then we may embed $A$ into a $\mathbb{Z}$-algebra with 1, namely $A' := A \times \mathbb{Z}$, where
addition is componentwise and multiplication is given by $(r,n)(s,m) = (rs + mr+ns, nm)$ \blu (see, e.g., \cite[p.~119, Thm.~1.10]{Hun80})\blk. The resulting ring with 1 is easily seen to be
torsion-free and
\[
\Id(A') = \{(e,0) \mid e \in \Id(A) \} \cup \{(-e, 1) \mid e \in \Id(A) \}.
\]
From this description $A'$ is generated by idempotents, so $A'$ is a Specker $\mathbb{Z}$-algebra, and hence $A'$ is isomorphic to
$C(Y, \mathbb{Z})$, where $Y$ is the Stone space of $\Id(A')$. The description of $\Id(A')$ shows that $\Id(A)$ embeds in $\Id(A')$
as a maximal ideal. Therefore, by Stone duality, it \blu corresponds to \blk the complement of a point $y \in Y$. Let $\alpha : A' \to C(Y, \mathbb{Z})$ be the canonical isomorphism
which sends $e \in \Id(A')$ to the characteristic function $\chi_{K_e}$, where $K_e$ is the clopen subset of $Y$ corresponding to $e$.
We show that $\alpha(A) = \{f \in C(Y, \mathbb{Z}) \mid f(y) = 0\}$, which is an $\ell$-ideal of $C(Y, \mathbb{Z})$. First, if
$e \in \Id(A)$, then $e \notin y$, and so $\chi_{K_e}(y) = 0$. Therefore, each $a \in A$ is sent to a function that vanishes at $y$.
Conversely, if $f \in C(Y, \mathbb{Z})$ with $f(y) = 0$, let $f = m_1\chi_{K_1} + \cdots + m_n \chi_{K_n}$ be an orthogonal representation.
Since $f(y) = 0$, we see that $y \notin K_i$ for each $i$. There are $e_i \in \Id(A')$ with $K_i = K_{e_i}$, and so each $e_i \notin y$.
This means each $e_i \in A$. Therefore, $f$ lies in the image of $A$. This shows that the image of $A$ is an $\ell$-ideal of
$C(Y, \mathbb{Z})$. Consequently, $G$ is a Specker $\ell$-group by Theorem~\ref{ideal}.
\qed
\end{proof}

Theorem~\ref{ig} gives  a description of  Specker $\ell$-groups purely in ring-theoretic terminology. 
Bergman \cite{Ber72} used a ring-theoretic approach to recover N\"{o}beling's theorem in the special case of a Specker $\ell$-group with a weak order-unit. 
We state his theorem in our terminology.

\begin{theorem} \label{Bergman}  {\em (Bergman \cite[Theorem 1.1]{Ber72})}
If $A$ is a Specker $\mathbb{Z}$-algebra, then $A$ is free as a ${\mathbb{Z}}$-module and has a basis consisting of idempotents.
\end{theorem}

Theorem~\ref{Bergman} is restricted to unital ${\mathbb{Z}}$-algebras and hence only gives
the freeness of Specker $\ell$-groups with a weak order-unit. However, combining Bergman's theorem with Corollary~\ref{ideal} and Theorem~\ref{ig}
 recovers N\"obeling's theorem that every Specker group is free with a basis of characteristic functions:

\begin{corollary}
If $G$ is a Specker $\ell$-group, then $G$ is a free abelian group having a basis of singular elements.
\end{corollary}

\begin{proof}
By Corollary~\ref{ideal}, there is a Specker $\ell$-group $H$ with a weak order-unit such that $G$ is \blu (isomorphic to) \blk an $\ell$-ideal
of $H$. By Theorem~\ref{ig}, $H$ is a Specker $\mathbb{Z}$-algebra with a weak order-unit, and then by Theorem~\ref{Bergman}, there is a basis $\mathcal{B}$ of $H$ consisting of singular elements. We claim that
\blu$G \cap \mathcal{B}$ \blk is a basis for $G$. Let $a \in G$ and first suppose that $0 < a$. Then we may write $a = m_1g_1 + \cdots + m_ng_n$
for some $g_i \in \mathcal{B}$ and some $0 < m_i$. We have $0 < g_i \le a$, so \blu $g_i \in G$. \blk Consequently, $a$ is a $\mathbb{Z}$-linear combination of
elements from \blu $G \cap \mathcal{B}$. \blk For an arbitrary $a \in G$, from the previous case both $a^+$ and $a^-$ are in the $\mathbb{Z}$-span of \blu $G \cap \mathcal{B}$. \blk Since $a = a^+ - a^-$, we see that $a$ is in the span of \blu $G \cap \mathcal{B}$. \blk Thus, \blu $G \cap \mathcal{B}$ \blk is
a basis for $G$ consisting of singular elements of $G$.
\qed
\end{proof}

 Bergman's proof of Theorem~\ref{Bergman} amounts to showing that for each Boolean algebra $B$, the abelian group ${\mathbb{Z}}[B]$ is free and has a basis of idempotents. Combining Theorem~\ref{main} with the observation that for a ring $R$ with 1, we have $R[B] \cong R \otimes_{\mathbb{Z}} {\mathbb{Z}}[B]$, we thus obtain that every Specker $R$-algebra is a free $R$-module.

It follows that if $R$ is a domain, then the idempotent-generated $R$-algebras $A$ that are Specker are precisely those that are free as $R$-modules (since in the latter case every nonzero idempotent is faithful). Less obviously, this remains true if $R$ is indecomposable, and that in this case, $A$ is projective iff $A$ is free:

\begin{theorem} \emph{\cite[Thm.~3.12]{BMMO15a}} \label{proj}
Let $R$ be indecomposable and let $A$ be an idempotent-generated $R$-algebra. Then the following are equivalent.
\begin{enumerate}[$(1)$]
\item  $A$ is a Specker $R$-algebra.
\item $A$ is a free $R$-module.
\item $A$ is a projective $R$-module.
\end{enumerate}
\end{theorem}

\begin{remark}
A converse of Theorem~\ref{proj} is also true: If $R$ is a ring, then $R$ is indecomposable iff for each idempotent-generated $R$-algebra $A$ that is projective as an $R$-module, $A$ is a Specker $R$-algebra; see \cite[Thm.~7.2]{BMO18}.  In \cite[Sec.~7]{BMO18}, further homological aspects of idempotent-generated algebras (e.g., flatness) are studied in relation to Specker $R$-algebras as well as to the locally Specker $R$-algebras discussed in Remark~\ref{rem: locally Specker}.
\end{remark}

\subsection{Specker ${\mathbb{R}}$-algebras}

In \cite[Def.~5.1]{BMO13a}, a {Specker ${\mathbb{R}}$-algebra} is defined to be a unital ${\mathbb{R}}$-algebra $A$ such that $A$ is generated as an ${\mathbb{R}}$-algebra by the  idempotents in $A$. Since  an ${\mathbb{R}}$-algebra is  torsion-free as an ${\mathbb{R}}$-module, this definition agrees with the definition of Specker algebra given in Section 3. In this section we use results from \cite{BMO13a} to illustrate the role that Specker ${\mathbb{R}}$-algebras play in the context of rings of continuous real-valued functions.

Let $A$ be a Specker ${\mathbb{R}}$-algebra. By Theorem~\ref{main}, there is a Stone space $X$, namely the Stone space of $\Id(A)$, such that $A$ is isomorphic as an ${\mathbb{R}}$-algebra to $C(X,{\mathbb{R}}_{\func{disc}})$. The discrete topology is somewhat unnatural when dealing with ${\mathbb{R}}$; instead we wish to consider ${\mathbb{R}}$ with the usual interval topology. By \cite[Prop.~5.4]{BMMO15a}, the fact that ${\mathbb{R}}$ is a totally ordered ring implies that $C(X,{\mathbb{R}}_{\func{disc}}) = FC(X,{\mathbb{R}})$, the ring of finitely-valued continuous functions into ${\mathbb{R}}$ with the usual interval topology. Since $X$ is compact, it follows that $FC(X,{\mathbb{R}})$ is the ring $PC(X,{\mathbb{R}})$ of piecewise constant real-valued continuous functions (see \cite[Rem.~5.5]{BMMO15a}). Thus, we obtain the following representation theorem for Specker ${\mathbb{R}}$-algebras.

\begin{theorem}  \emph{\cite[Thm.~6.2(6)]{BMO13a}}\label{bal}
An ${\mathbb{R}}$-algebra $A$ is  a Specker ${\mathbb{R}}$-algebra iff $A$ is isomorphic as an ${\mathbb{R}}$-algebra to $PC(X, {\mathbb{R}})$ for some Stone space $X$.
\end{theorem}

The theorem implies that every Specker ${\mathbb{R}}$-algebra $A$ admits an $\ell$-algebra structure.
Since the isomorphism in the theorem sends idempotents to idempotents, it follows that the isomorphism carries the set of idempotents of $A$ onto the set of characteristic functions on $X$, and hence the $\ell$-algebra structure on $A$ extends the usual order on the Boolean algebra of idempotents of $A$.

The next theorem gives a ring-theoretic characterization of Specker $\mathbb{R}$-algebras. For this we recall that a ring $A$ is a \emph{von Neumann regular ring} if for all $ a\in A$ there is $b \in A$ with $aba=a$.

\begin{theorem} \emph{\cite[Thm.~6.2(2)]{BMO13a}} \label{Specker char}
Let $A$ be an $\ell$-subalgebra of $C(X, \mathbb{R})$ for some compact Hausdorff space $X$. Then $A$ is a Specker $\mathbb R$-algebra iff $A$ is a von Neumann regular ring.
\end{theorem}

\begin{remark}
Let $X$ be a compact Hausdorff space. The uniform norm on $C(X, \mathbb{R})$ is given by $\|f\| = \sup\{ |f(x)| \mid x \in X\}$. If $X$ is a Stone space, then $PC(X,{\mathbb{R}})$ separates points of $X$, hence by the Stone-Weierstrass theorem, $PC(X, \mathbb{R})$ is uniformly dense in $C(X,{\mathbb{R}})$. Thus, Specker ${\mathbb{R}}$-algebras can be viewed as encoding the essential algebraic data of rings of continuous real-valued functions on Stone spaces. In fact, it follows from \cite[Secs.~5 and 6]{BMO13a} that for a Stone space $X$ the Specker $\mathbb{R}$-algebra $PC(X,{\mathbb{R}})$ is the smallest uniformly dense $\ell$-subalgebra of $C(X, \mathbb{R})$. For a more formal treatment of this idea in terms of categories and coreflectors see \cite[Sec.~6]{BMO13a}.
\end{remark}

\begin{remark}
If an $\ell$-group $G$ has an $\mathbb{R}$-vector space structure compatible with the order, then $G$ is called a \emph{vector lattice} or \emph{Riesz space}. These structures have been investigated in detail (see, e.g., \cite{LZ71}). A Specker $\mathbb{R}$-algebra $A$ viewed as a vector lattice is \emph{hyperarchimedean}, meaning that $A/I$ is archimedean for each $\ell$-ideal $I$ of $A$ (the converse is also true; see \cite[Thm.~6.2]{BMO13a}). Hyperarchimedean vector lattices with a weak order-unit were studied in \cite{BM14}, where a result analogous to Theorem~\ref{bal} was proved. It follows that the category of Specker $\mathbb{R}$-algebras is equivalent to the category of hyperarchimedean vector lattices with a weak order-unit.
\end{remark}

\def\cprime{$'$}

\end{document}